\documentclass[11pt]{amsart}
\usepackage{amsmath}
\usepackage{amssymb}
\usepackage{amsfonts}
\usepackage{amsthm}
\usepackage{enumerate}
\usepackage[mathscr]{eucal}
\usepackage{verbatim}
\usepackage{amsthm}
\usepackage{amscd}
\usepackage{appendix}
\usepackage{tikz}

\newtheorem{theorem}{Theorem}[section]

\newtheorem{lemma}[theorem]{Lemma}

\theoremstyle{definition}

\newtheorem{remark}{Remark}

\newtheorem{innercustomgeneric}{\customgenericname}
\providecommand{\customgenericname}{}

\newcommand{\newcustomtheorem}[2]{\newenvironment{#1}[1]
  {\renewcommand\customgenericname{#2}
   \renewcommand\theinnercustomgeneric{##1}\innercustomgeneric}{\endinnercustomgeneric}}

\newcustomtheorem{customthm}{Theorem}

\newcommand{\newcustomlemma}[2]{\newenvironment{#1}[1]
  {\renewcommand\customgenericname{#2}
   \renewcommand\theinnercustomgeneric{##1} \innercustomgeneric}{\endinnercustomgeneric}}

\newcustomlemma{customlemma}{Lemma}

\newcustomlemma{customproposition}{Proposition}

\numberwithin{equation}{section}

\newcommand{\rr}{\mathbb{R}}
\newcommand{\nn}{\mathbb{N}}

\newcommand{\rn}{\mathbb{R}^n}

\newcommand{\zz}{\mathbb{Z}}
\newcommand{\zn}{\mathbb{Z}^n}

\def\SS{{\mathscr{S}}}

\def\xxxi{\vec{\boldsymbol{\xi}}}
\def\000{\vec{\boldsymbol{0}}}
\def\|{{\boldsymbol{|}}}

\def\fff{\vec{\boldsymbol{f}}}

\def\xxx{\vec{\boldsymbol{x}}}
\def\yyy{\vec{\boldsymbol{y}}}

\newcommand{\wh}{\widehat}

\begin{document}

\author{Loukas Grafakos}
\address{L. Grafakos, Department of Mathematics, University of Missouri, Columbia, MO 65211, USA} 
\email{grafakosl@missouri.edu}

\author{Bae Jun Park}
\address{B. Park, School of Mathematics, Korea Institute for Advanced Study, Seoul 02455, Republic of Korea}
\email{qkrqowns@kias.re.kr}

\thanks{The first author would like to acknowledge the support of  the Simons Foundation grant 624733. The second author is supported in part by NRF grant 2019R1F1A1044075 and by a KIAS Individual Grant MG070001 at the Korea Institute for Advanced Study}

\title[Multilinear multiplier theorem]{The multilinear H\"ormander multiplier theorem with a Lorentz-Sobolev  condition}

\begin{abstract} 
In this article, we provide a multilinear  version  of the H\"ormander multiplier theorem 
with a Lorentz-Sobolev space condition. The work is motivated by the recent result  of the first author and Slav\'ikov\'a \cite{Gr_Sl} where an analogous version of classical H\"ormander multiplier theorem was obtained;  
this version is sharp in many ways and reduces the number of indices that appear in the statement of the theorem.  As a natural extension of the linear case, in this work,  we prove that if $mn/2<s<mn$, then
\begin{equation*}
\big\Vert T_{\sigma}(f_1,\dots,f_m)\big\Vert_{L^p(\rn)}\lesssim \sup_{k\in\zz}\big\Vert \sigma(2^k\;\vec{\cdot}\;)\wh{\Psi^{(m)}}\big\Vert_{L_{s}^{mn/s,1}(\rr^{mn})}\Vert f_1\Vert_{L^{p_1}(\rn)}\cdots \Vert f_m\Vert_{L^{p_m}(\rn)} 
\end{equation*}
 for certain $p,p_1,\dots,p_m$ with $1/p=1/p_1+\dots+1/p_m$. We also show that the above estimate is sharp, in the sense that the Lorentz-Sobolev space $L_s^{mn/s,1}$ cannot be replaced by $L_{s}^{r,q}$ for $r<mn/s$, $0<q\leq \infty$, or by $L_s^{mn/s,q}$ for $q>1$.

\end{abstract}

\maketitle

\section{Introduction}\label{introduction}

Let $\mathscr{S}(\rn)$ denote the space of all Schwartz functions on $\rn$.
Given a bounded function $\sigma$ on $\rn$, we define a linear multiplier operator
\begin{equation*}
T_{\sigma}f(x):=\int_{\rn}{\sigma(\xi)\wh{f}(\xi)e^{2\pi i\langle x,\xi\rangle}}d\xi
\end{equation*} acting on  $f\in\SS(\rn)$ where $\wh{f}(\xi):=\int_{\rn}{f(x)e^{-2\pi i\langle x,\xi\rangle}}dx$ is the Fourier transform of $f$.
One of important problems in harmonic analysis is to find optimal sufficient conditions on $\sigma$ for the corresponding operator $T_{\sigma}$ to admit an $L^p$-bounded extension for all $1<p<\infty$.
The classical theorem of Mikhlin \cite{Mik} states that if the condition
\begin{equation*}
\big| \partial_{\xi}^{\alpha}\sigma(\xi)\big|\lesssim_{\alpha}|\xi|^{-|\alpha|}, \quad \xi\not= 0
\end{equation*} holds for all multi-indices $\alpha$ with $|\alpha|\leq [n/2]+1$, then
 $T_{\sigma}$ extends to a bounded operator in $L^p$ for $1<p<\infty$. 
 H\"ormander \cite{Ho} refined this result, using the weaker condition
\begin{equation}\label{Hocondition}
\sup_{k\in\zz}{\big\Vert \sigma(2^k\cdot)\wh{\psi}\big\Vert_{L_s^2(\rn)}}<\infty
\end{equation} for $s>n/2$, where $L_s^2(\rn)$ denotes the standard fractional Sobolev space on $\rn$ and $\psi$ is a Schwartz function on $\rn$ whose Fourier transform is supported in the annulus $1/2<|\xi|<2$ and satisfies $\sum_{k\in\zz}{\wh{\psi}(\xi/2^k)}=1$ for $\xi\not= 0$.
Calder\'on and Torchinsky \cite{Ca_To} proved that if (\ref{Hocondition}) holds for $s>n/p-n/2$, then $T_{\sigma}$ is bounded in $H^p(\rn)$ for $0<p\leq 1$.
They also showed that $L_s^2$ in (\ref{Hocondition}) can be replaced by $L_s^r$ for the $L^p$-boundedness, $1<p<\infty$, using a complex interpolation method, and the assumption in their result was weakened by Grafakos, He, Honz\'ik, and Nguyen \cite{Gr_He_Ho_Ng}. 
Recently, Grafakos and Slav\'ikov\'a \cite{Gr_Sl} have improved the previous multiplier theorems by replacing $L_s^r$ by the Lorentz-Sobolev space $L_s^{n/s,1}$.

We recall the definition of Lorentz spaces $L^{p,q}(\rn)$ and Lorentz-Sobolev spaces $L^{p,q}_s(\rn)$.
For any measurable function $f$ on $\rn$, we let $d_f(s):=\big| \{x\in\rn:|f(x)|>s\}\big|$ be the distribution function of $f$ and 
\begin{equation*}
f^*(t):=\inf\big\{s>0: d_f(s)  \leq t \big\}, \qquad  t>0
\end{equation*} 
be its decreasing rearrangement.
We adopt the convention that the infimum of the empty set is $\infty$.
For $0<p,q\leq \infty$ the quasi-norm on the Lorentz space $L^{p,q}(\rn)$ is given by
\begin{equation*}
\Vert f\Vert_{L^{p,q}(\rn)}:=\begin{cases}
\displaystyle \Big(\int_0^{\infty}{\big(t^{1/p}f^*(t) \big)^{q}}\frac{dt}{t} \Big)^{1/q}, & q<\infty\\
\qquad \displaystyle\sup_{t>0}{t^{1/p}f^*(t)}, & q=\infty.
\end{cases}
\end{equation*} 
For $s>0$ let $(I-\Delta)^{s/2}$ be the inhomogeneous fractional Laplacian operator, explicitly defined by
\begin{equation*}
(I-\Delta)^{s/2}f:=\big( (1+4\pi^2|\cdot|^2)^{s/2}\widehat{f}\big)^{\vee}
\end{equation*} 
where $f^{\vee}(\xi):=\wh{f}(-\xi)$ is the inverse Fourier transform of $f$. 
Then for $0<p,q\leq \infty$ and $s>0$ we define
\begin{equation*}
\Vert f\Vert_{L^{p,q}_s(\rn)}:=\big\Vert (I-\Delta)^{s/2}f\big\Vert_{L^{p,q}(\rn)}.
\end{equation*}

\begin{customthm}{A}(\cite{Gr_Sl})\label{knownresult}
Let $1<p<\infty$ and $|n/p-n/2|<s<n$.
Then there exists $C>0$ such that
\begin{equation}\label{knownest}
\Vert T_{\sigma}f\Vert_{L^p(\rn)}\leq C \sup_{k\in \zz}{\big\Vert \sigma(2^k\cdot)\wh{\psi}\big\Vert_{L_s^{n/s,1}(\rn)}}\Vert f\Vert_{L^p(\rn)}.
\end{equation}
\end{customthm}

We also refer to \cite{Gr_Park} for an extension of Theorem \ref{knownresult} to the Hardy space $H^p(\rn)$ for $0<p<\infty$. 
Note that for $0<r_1< r_2<\infty$ and $0<q_1,q_2\leq \infty$
\begin{equation}\label{embedding1}
\big\Vert \sigma(2^k\cdot)\wh{\psi}\big\Vert_{L_s^{r_1,q_1}(\rn)}\lesssim \big\Vert \sigma(2^k\cdot)\wh{\psi}\big\Vert_{L_{s}^{r_2,q_2}(\rn)} \quad \text{~uniformly in }~ k,
\end{equation}
which follows from H\"older's inequality with even integers $s$, complex interpolation technique, and a proper embedding theorem. 
Moreover, if $q_1\geq q_2$, then the embedding $L_s^{r,q_2}(\rn) \hookrightarrow L_s^{r,q_1}(\rn)$ yields that 
\begin{equation}\label{embedding2}
\big\Vert \sigma(2^k\cdot)\widehat{\psi}\big\Vert_{L_s^{r,q_1}(\rn)}\lesssim \big\Vert \sigma(2^k\cdot)\widehat{\psi}\big\Vert_{L_{s}^{r,q_2}(\rn)} \quad \text{~uniformly in }~ k. 
\end{equation}
 Thus,  $L_s^{n/s,1}(\rn)$ is   bigger   than $L_s^{r,q}(\rn)$ for $r> n/s$  when $0<q\leq \infty$ and than  $L_s^{n/s,q}(\rn)$ when $0<q< 1$; the spaces $L^r_s(\rn)=L^{r,r}_s(\rn)$ with $r>n/s$ appeared in 
previous versions of  the H\"ormander multiplier theorem. Moreover, it was shown in  
  \cite{Gr_Park}  that the parameters $r=n/s$ and $q=1$ in   Theorem \ref{knownresult} are sharp, i.e., 
  boundedness   in   (\ref{knownest}) fails if 
  $n/s$ is replaced $r<n/s$ or if $1$ is replaced    by      $q>1$.

\hfill

We now turn our attention to   multilinear  multiplier theory, which is the focus of this paper.
 Let $m$ be a positive integer greater than $1$, which will serve as the degree of the multilinearity of operators.
For a bounded function $\sigma$ on $\rr^{mn}$ we define the corresponding $m$-linear multiplier operator $T_{\sigma}$ by
\begin{equation*}
T_{\sigma}\big(f_1,\dots,f_m \big)(x):={\int_{\rr^{mn}}{\sigma(\xxxi)\Big(\prod_{j=1}^{m}\widehat{f_j}(\xi_j)\Big)e^{2\pi i\langle x,\sum_{j=1}^{m}{\xi_j} \rangle}}d\xxxi},\qquad x\in\rn
\end{equation*} for $f_j\in \SS(\rn)$ where $\xxxi:=(\xi_1,\dots,\xi_m)\in (\rn)^m$ and $d\xxxi:=d\xi_1\cdots d\xi_m$.
As a multilinear extension of Mikhlin's result, Coifman and Meyer \cite{Co_Me2} proved that if $L$ is sufficiently large and $\sigma$ satisfies
\begin{equation*}
\big| \partial_{\xi_1}^{\alpha_1} \cdots\partial_{\xi_m}^{\alpha_m}\sigma(\xxxi)\big|\lesssim_{\alpha_1,\dots,\alpha_m}\big(|\xi_1|+\dots+|\xi_m|\big)^{-(|\alpha_1|+\dots +|\alpha_m|)}, \quad \xxxi\not= \000
\end{equation*}
for $\xi_1,\dots,\xi_m\in\rn$ and multi-indices $\alpha_1,\dots,\alpha_m\in \zn$ with $|\alpha_1|+\dots+|\alpha_m|\leq L$, then $T_{\sigma}$ is bounded from $L^{p_1}\times \cdots\times L^{p_m}$ to $L^p$ for all $1<p<\infty$ and $1<p_1,\dots,p_m\leq \infty$ satisfying $1/p=1/p_1+\dots+1/p_m$. This result was extended to $p\leq 1$ by Kenig and Stein \cite{Ke_St} and Grafakos and Torres \cite{Gr_To}.

Let $\Psi^{(m)}$ be the $m$-linear counterpart of $\psi$. That is, $\Psi^{(m)}$ is a Schwartz function on $\rr^{mn}$ having the properties:
\begin{equation*}
\textup{Supp}(\widehat{\Psi^{(m)}})\subset \big\{\xxxi\in \rr^{mn}: 1/2\leq |\xxxi|\leq 2 \big\}, \qquad \sum_{k\in\zz}{\widehat{\Psi^{(m)}}(\xxxi/2^k)}=1, \quad \xxxi\not= \000.
\end{equation*}
Let $(\vec{I}-\vec{\Delta})^{s/2}$ denote the inhomogeneous
 fractional Laplacian operator acting on functions on $\rr^{mn}$.
For $s\geq 0$ and $0<r<\infty$ the Sobolev norm of $f$ is defined as
\begin{equation*}
\Vert f\Vert_{L_s^r(\rr^{mn})}:=\big\Vert (\vec{I}-\vec{\Delta})^{s/2}f \big\Vert_{L^r(\rr^{mn})}.
\end{equation*} 
Tomita \cite{Tom} obtained an $L^{p_1}\times \cdots\times L^{p_m}$  to $L^p$ boundedness for $T_{\sigma}$ in the range $1<p,p_1,\dots,p_m<\infty$ under a condition analogous to (\ref{Hocondition}):
\begin{customthm}{B}(\cite{Tom}) \label{knownresult1}
Let $1<p,p_1,\dots,p_m<\infty$  satisfy $1/p=1/p_1+\dots+1/p_m$.
Suppose  $s>mn/2$.
Then there exists $C>0$ such that
\begin{equation*}
\big\Vert T_{\sigma}\big(f_1,\dots,f_m \big) \big\Vert_{L^p(\rn)}\leq C \sup_{k\in \zz}\big\Vert \sigma(2^k\;\vec{\cdot}\;)\wh{\Psi^{(m)}}\big\Vert_{L^{2}_s(\rr^{mn})} \prod_{j=1}^{m}{\Vert f_j\Vert_{L^{p_j}(\rn)}}
\end{equation*}
for $f_1,\dots,f_m\in \SS(\rn)$.
\end{customthm} 

Grafakos and Si \cite{Gr_Si} extended Theorem \ref{knownresult1} to $p\leq 1$ using $L^r$-based Sobolev norms of $\sigma$ for $1<r\leq 2$:
\begin{customthm}{C}(\cite{Gr_Si}) \label{knownresult2}
Let $1<r\leq 2$, $0<p<\infty$, $r\leq p_1,\dots,p_m<\infty$, and $1/p_1+\dots+1/p_m=1/p$. 
Suppose $s>mn/r$.
Then there exists $C>0$ such that
\begin{equation*}
\big\Vert T_{\sigma}\big(f_1,\dots,f_m \big) \big\Vert_{L^p(\rn)}\leq C \sup_{k\in \zz}\big\Vert \sigma(2^k\;\vec{\cdot}\;)\wh{\Psi^{(m)}}\big\Vert_{L^{r}_s(\rr^{mn})} \prod_{j=1}^{m}{\Vert f_j\Vert_{L^{p_j}(\rn)}}
\end{equation*}
for  $f_1,\dots,f_m\in \SS(\rn)$.
\end{customthm} 

Note that Theorem \ref{knownresult2} provides a broader range of $p$'s but requires stronger assumptions on $s$, while, under the same condition $s>mn/2$ (when $r=2$), the estimate in Theorem \ref{knownresult2} is contained in Theorem \ref{knownresult1}.
We also refer to \cite{Fu_Tom, Gr_He_Ho, Gr_Mi_Ng_Tom, Gr_Ng, Gr_Park2, Mi_Tom, Park2, Tom} for further results.

\hfill

The aim of this paper is to provide a multilinear extension of Theorem \ref{knownresult}, which also provides
 a sharp version of Theorem \ref{knownresult1} and \ref{knownresult2}.
In order to state our main results, we first define two open sets $\mathcal{Q}_l$ and $\mathcal{P}$ in $\rr^m$ as follows
\begin{equation*}
\mathcal{Q}_l:=\big\{(r_1,\dots,r_m)\in \rr^m:0<r_j<l,~ 1\leq j\leq m \big\},
\end{equation*}
\begin{equation*}
\mathcal{P}:=\big\{(r_1,\dots,r_m)\in \rr^m: 0<r_1+\dots+r_m<1 \big\},
\end{equation*}
and denote by $hull{\big( \mathcal{Q}_{l},\mathcal{P}\big)}$ the convex hull containing both $\mathcal{Q}_l$ and $\mathcal{P}$.
Then our first  main result  is
\begin{theorem}\label{main1}
Let $0<p,p_1,\dots,p_m<\infty$ satisfy $1/p=1/p_1+\dots+1/p_m$. Suppose  $mn/2<s<mn$ and 
\begin{equation*}
(1/p_1,\dots,1/p_m)\in hull{\big( \mathcal{Q}_{\frac{s}{mn}},\mathcal{P}\big)}.
\end{equation*}
Then there exists $C>0$ such that
\begin{equation}\label{mainresult}
\big\Vert T_{\sigma}\fff\big\Vert_{L^p(\rn)}\leq C \sup_{k\in\zz}\big\Vert  \sigma(2^k\;\vec{\cdot}\;)  \widehat{\Psi^{(m)}}\big\Vert_{L_{s}^{mn/s,1}(\rr^{mn})}\prod_{j=1}^{m}{\Vert f_j\Vert_{L^{p_j}(\rn)}}
\end{equation} for $f_1,\dots, f_m \in \SS(\rn)$.
\end{theorem}
 Figure \ref{fig1}  shows the range of indices $p_1,p_2$ for which boundedness holds in the bilinear case $m=2$.
Note that only two cases $(1/p_1,\dots,1/p_m)\in \mathcal{Q}_{\frac{s}{mn}}$ and $(1/p_1,\dots,1/p_m)\in \mathcal{P}$ will be treated in the proof
as the desired result follows immediately via interpolation.
The first case is equivalent to $mn/s<p_1,\dots,p_m<\infty$ for which the proof is based on the Littlewood-Paley theory and the pointwise estimate in Lemma \ref{keylemma1} below. 
Since $mn/s<2$, the first one contains the result for $2\leq p_1,\dots,p_m<\infty$ and then a method of transposes of $T_{\sigma}$ and duality arguments will be applied to the case $1<p,p_1,\dots,p_m<\infty$ that coincides with the second part.

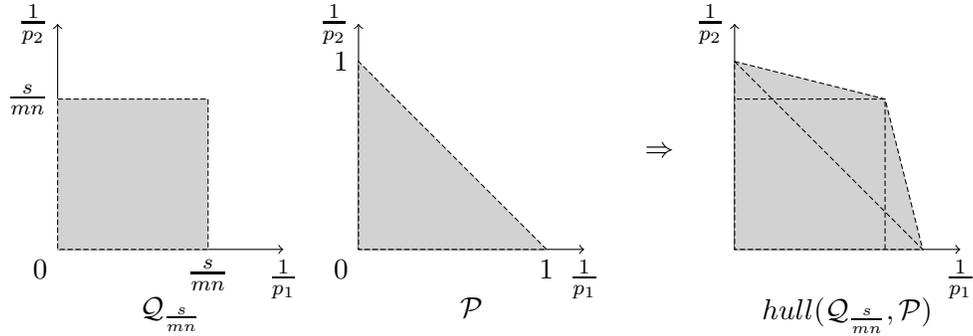
\begin{figure}[h]\label{fig1}
\begin{tikzpicture}
\draw [<->] (-3,3)--(-3,0)--(0,0);
\draw [<->] (1,3)--(1,0)--(4,0);
\draw [<->] (6,3)--(6,0)--(9,0);
\path[fill=gray!35] (-3,2)--(-3,0)--(-1,0)--(-1,2);
\path[fill=gray!35] (1,0)--(3.5,0)--(1,2.5);
\path[fill=gray!35] (6,0)--(8.5,0)--(8,2)--(6,2.5);
\draw[dash pattern= { on 2pt off 1pt}] (-3,2)--(-1,2)--(-1,0)--(-3,0)--(-3,2);
\draw[dash pattern= { on 2pt off 1pt}] (3.5,0)--(1,2.5)--(1,0)--(3.5,0);
\draw[dash pattern= { on 2pt off 1pt}]  (6,2.5)--(8.5,0)--(8,2)--(6,2.5)--(6,0)--(8.5,0); 
\draw[dash pattern= { on 2pt off 1pt}]  (8,0)--(8,2)--(6,2); 
\node [below left] at (-3,0) {$0$};\node [left] at (-3,2) {$\frac{s}{mn}$};\node [below] at (-1,0) {$\frac{s}{mn}$}; \node [below] at (0,0) {$\frac{1}{p_1}$};\node [left] at (-3,3) {$\frac{1}{p_2}$};
\node [below left] at (1,0) {$0$};\node [left] at (1,2.5) {$1$};\node [below] at (3.5,0) {$1$};\node [below] at (4,0) {$\frac{1}{p_1}$};\node [left] at (1,3) {$\frac{1}{p_2}$};
\node [below] at (-1.5,-0.5) {$\mathcal{Q}_{\frac{s}{mn}}$};\node [below] at (2.5,-0.5) {$\mathcal{P}$};\node [below] at (7.5,-0.5) {$hull(\mathcal{Q}_{\frac{s}{mn}}, \mathcal{P})$};\node [below] at (9,0) {$\frac{1}{p_1}$};\node [left] at (6,3) {$\frac{1}{p_2}$};
\node [below] at (5,1.5) {$\Rightarrow$};
\end{tikzpicture}
\caption{$L^{p_1}\times L^{p_2}\to L^p$ boundedness of $T_{\sigma}$, $m=2$.}
\end{figure}

As in the linear case, using   in (\ref{embedding1}) and (\ref{embedding2}), we may   replace $L_s^{mn/s,1}$ in (\ref{mainresult}) by $L_s^{r,q}$ for $r>mn/s$ and $0<q\leq \infty$ or by $L_s^{mn/s,q}$ for $0<q<1$.
We remark that Theorem \ref{main1} clearly improves Theorem \ref{knownresult1} and \ref{knownresult2} in view of $L_s^r=L_s^{r,r}$.
 
Our second main result is the sharpness of the parameters $r,q$.
That is,  $r=mn/s$ cannot be replaced by a smaller number, and if $r=mn/s$, then $q=1$ is the largest number for (\ref{mainresult}) to hold. 
This is contained in the following theorem:
\begin{theorem}\label{main2}
Let $0<p<\infty$ and $0< p_1,\dots,p_m\leq \infty$ satisfy $1/p=1/p_1+\dots+1/p_m$. Suppose  $0<s<mn$. 
\begin{enumerate}
\item For any $0<r<mn/s$ and $0<q\leq \infty$, there exists $\sigma$ satisfying
\begin{equation}\label{necessary}
\sup_{k\in\zz}\big\Vert  \sigma(2^k\;\vec{\cdot}\;) \widehat{\Psi^{(m)}}\big\Vert_{L_{s}^{r,q}(\rr^{mn})}<\infty
\end{equation}
such that $T_{\sigma}$ is not bounded from $L^{p_1}\times\cdots\times L^{p_m}$ to $L^p$.

\item For $q>1$, there exists $\sigma$ satisfying 
\begin{equation*}
\sup_{k\in\zz}\big\Vert  \sigma(2^k\;\vec{\cdot}\;) \widehat{\Psi^{(m)}}\big\Vert_{L_{s}^{mn/s,q}(\rr^{mn})}<\infty
\end{equation*}
such that $T_{\sigma}$ is not bounded from $L^{p_1}\times\cdots\times L^{p_m}$ to $L^p$.
\end{enumerate}
\end{theorem}
The key ingredients in the proof of Theorem \ref{main2} are a variant of Bessel potential estimates introduced by Grafakos and Park \cite{Gr_Park}, and the scaling arguments used in \cite{Gr_Park2, Park3}.

\begin{remark}
 Theorem \ref{main2} proves that (\ref{necessary}) with $r\geq mn/s$ and $q\leq 1$ is a necessary condition for the boundedness of $T_{\sigma}$ for all $0<p<\infty$ and $0<p,p_1,\dots,p_m\leq \infty$ satisfying $1/p=1/p_1+\dots +1/p_m$. Since our techniques are not applicable to the case $(1/p_1,\dots,1/p_m)\notin hull( \mathcal{Q}_{\frac{s}{mn}},\mathcal{P})$ in Theorem \ref{main1}, an alternative argument will be needed in the case.
 A similar question arises in terms of the parameter $s$. That is, we need to verify that estimate (\ref{mainresult}) holds for $0<s\leq mn/2$.

\end{remark}

\section{Preliminaries : Inequalities in Lorentz spaces}\label{preliminary}
In this section we review several inequalities that will be useful in the proof of the main results, and refer the reader to \cite{Gr_Park, Gr_Sl}.
We fix $N\in\nn$ and discuss inequalities of functions on the $N$ dimensional space $\rr^N$.

For a locally integrable function $f$ defined on $\rr^N$, let 
$$
\mathcal{M}^{(N)}f(x):=\sup_{Q:x\in Q}\frac{1}{|Q|}\int_Q{|f(y)|}dy
$$
 be the Hardy-Littlewood maximal function of $f$ where the supremum is taken over all cubes in $\rr^N$ containing $x$, and $\mathcal{M}^{(N)}_rf(x):=\big( \mathcal{M}^{(N)}\big(|f|^r\big)(x) \big)^{1/r}$ for $0<r<\infty$.
Then the Fefferman-Stein vector-valued maximal inequality \cite{Fe_St} says that for $0<r<p,q<\infty$
\begin{equation}\label{fsmaximal}
\big\Vert \big\{\mathcal{M}_r^{(N)}f_k\big\}_{k\in\zz} \big\Vert_{L^p(\ell^q)}\lesssim \Vert \{f_k\}_{k\in\zz}\Vert_{L^p(\ell^q)}.
\end{equation}
Moreover, (\ref{fsmaximal}) holds for $0<p\leq \infty$ and $q=\infty$.

We now recall some inequalities in Lorentz spaces. 
Most of them are consequences of a real interpolation technique and inequalities in Lebesgue spaces.

\begin{lemma}\cite[Lemma 2.1]{Gr_Park}\label{young}
Let $1<p\leq r<\infty$, $1\leq q<r$, and $0<t\leq \infty$ satisfy $1/r+1=1/p+1/q$.
Then
\begin{equation*}
\Vert f\ast g\Vert_{L^{r,t}(\rr^N)}\leq \Vert f\Vert_{L^{p,t}(\rr^N)} \Vert g\Vert_{L^q(\rr^N)}.
\end{equation*}
\end{lemma}

\begin{lemma}\cite[Lemma 2.2]{Gr_Park}\label{hausyoung}
Let $2<p<\infty$ and $0<r\leq \infty$.
Then 
\begin{equation*}
\Vert \widehat{f}\, \Vert_{L^{p,r}(\rr^N)}\leq \Vert f\Vert_{L^{p',r}(\rr^N)},
\end{equation*}
where $1/p+1/p'=1$.

\end{lemma}

\begin{lemma}\cite[Lemma 2.3]{Gr_Park}\label{katoponce}
Let $1<p<\infty$, $0<r\leq \infty$, and $s>0$. For any $\vartheta\in S(\rr^{N})$, we have
\begin{equation*}
\Vert  \vartheta\cdot f \Vert_{L^{p,r}_{s}(\rr^N)}\lesssim_{N,s,p,r,\vartheta} \Vert f\Vert_{L_s^{p,r}(\rr^N)}.
\end{equation*}
\end{lemma}

\begin{lemma}\cite[Lemma 2.5]{Gr_Park}\label{holder}
Let $1<p<\infty$ and $1\leq q\leq \infty$.
Then
\begin{equation*}
 \int_{\rr^N}{\big|f(x)g(x)\big|}dx\leq \Vert f\Vert_{L^{p,q}(\rr^N)}\Vert g\Vert_{L^{p',q'}(\rr^N)}
\end{equation*}
where $1/p+1/p'=1/q+1/q'=1$.
\end{lemma}

A significant role is played in the proof of the main theorem by the following lemma, whose proof can be found in \cite{Gr_Sl}.
\begin{lemma}\cite[Lemma 2.1]{Gr_Sl}\label{grsllemma}
Let $0<s<N$, and $q>N/s$. Then for any measurable function $f$ on $\rr^N$ and $k\in\zz$, there exists $C>0$ such that 
\begin{equation*}
\Big\Vert \frac{f(x-\cdot/2^k) }{(1+4\pi^2|\cdot|^2)^{s/2}}\Big\Vert_{L^{N/s,\infty}(\rr^N)}\leq C \mathcal{M}_q^{(N)}f(x) \qquad \text{uniformly in }~ k.
\end{equation*}

\end{lemma}

\section{Proof of Theorem \ref{main1}}\label{proofmain1}

\subsection{The case $mn/s< p_1,\dots,p_m<\infty$}\label{firstcase}

Let $\Theta^{(m)}$ be a Schwartz function on $\rr^{mn}$ such that
\begin{equation*}
 \widehat{\Theta^{(m)}}(\xxxi)=1 \quad \text{ for }~ 2^{-2}m^{1/2}\leq |\xxxi|\leq 2^2 m^{1/2}, 
\end{equation*}
\begin{equation*}
\textup{Supp}(\widehat{\Theta^{(mn)}})\subset \big\{\xxxi \in \rr^{mn}:2^{-3}m^{-1/2}\leq |\xxxi|\leq 2^3m^{1/2} \big\}.
\end{equation*}
Using the fact that $\sum_{k\in\zz}{\widehat{\Psi^{(m)}}(\xxxi/2^k)}=1$ for $\xxxi\not= \000$, a triangle inequality, and Lemma \ref{katoponce}, we see that 
\begin{equation}\label{reproducing}
\sup_{k\in\zz}{\big\Vert \sigma(2^k\;\vec{\cdot}\;) \wh{\Theta^{(m)}}\big\Vert_{L_s^{mn/s,1}(\rr^{mn})}}\lesssim \sup_{k\in\zz}{\big\Vert \sigma (2^k\;\vec{\cdot}\;) \wh{\Psi^{(m)}}\big\Vert_{L_s^{mn/s,1}(\rr^{mn})}}.
\end{equation}
Therefore it suffices to show that
\begin{equation}\label{finalgoal}
\big\Vert T_{\sigma}\fff\big\Vert_{L^p(\rn)}\lesssim \sup_{k\in\zz}{\big\Vert \sigma (2^k\;\vec{\cdot}\;) \wh{\Theta^{(m)}}\big\Vert_{L_s^{mn/s,1}(\rr^{mn})}} \prod_{j=1}^{m}{\Vert f_j\Vert_{L^{p_j}(\rn)}}.
\end{equation}

We split $\sigma$, via the Littlewood-Paley partition of unity $\{\psi_k\}_{k\in\mathbb{Z}}$, as
\begin{align*}
\sigma&=\sum_{k_1,\dots,k_m \in \mathbb{Z}}{\sigma \cdot \big(\widehat{\psi_{k_1}}\otimes \cdots \otimes\widehat{\psi_{k_m}}\big)}\\
  &=\Big(\sum_{k_1\in\mathbb{Z}}\sum_{k_2,\dots,k_m\leq k_1}{\cdots}\Big)+\Big(\sum_{k_2\in\mathbb{Z}}\sum_{\substack{k_1<k_2\\k_3,\dots,k_m\leq k_2}}{\cdots}\Big)+\dots+ \Big(\sum_{k_m\in\mathbb{Z}}\sum_{k_1,\dots,k_{n-1}<k_m}{\cdots}\Big)\\
  &=:\sigma^{(1)}+\sigma^{(2)}+\dots+\sigma^{(m)}.
\end{align*}
Then (\ref{finalgoal}) is a consequence of the estimates
\begin{equation*}
\big\Vert T_{\sigma^{(l)}}\fff \big\Vert_{L^p(\rn)}\lesssim \sup_{k\in\zz}{\big\Vert \sigma (2^k\;\vec{\cdot}\;) \wh{\Theta^{(m)}}\big\Vert_{L_s^{mn/s,1}(\rr^{mn})}}\prod_{j=1}^{m}\Vert f_j\Vert_{L^{p_j}(\rn)}, \qquad 1\leq l\leq m.
\end{equation*} 
We are only concerned with the case $l=1$ as the others   follow from symmetric arguments.
We write 
\begin{align*}
\sigma^{(1)}(\xxxi)&=\sum_{k\in\mathbb{Z}}\sum_{k_2,\dots,k_m\leq k}{\sigma(\xxxi)\widehat{\psi_k}(\xi_1)\widehat{\psi_{k_2}}(\xi_2)\cdots\widehat{\psi_{k_m}}(\xi_m)}\\
 &=\sum_{k\in\mathbb{Z}}{\sigma(\xxxi)\widehat{\Theta^{(m)}}(\xxxi/2^k) \widehat{\psi_k}(\xi_1)\sum_{k_2,\dots,k_m\leq k}\widehat{\psi_{k_2}}(\xi_2)\cdots\widehat{\psi_{k_m}}(\xi_m)}
\end{align*} 
since $\widehat{\Theta^{(m)}}(\xxxi/2^k)=1$ for $ 2^{k-1}\leq |\xi_1|\leq 2^{k+1}$ and $|\xi_j|\leq 2^{k+1}$ for $2\leq j\leq m$.
Let 
\begin{equation}\label{sigmak}
\sigma_k(\xxxi):=\sigma(\xxxi)\widehat{\Theta^{(m)}}(\xxxi/2^k).
\end{equation}
Then we have
\begin{equation*}
\sigma^{(1)}(\xxxi)=\sum_{k\in\zz}{\sigma_k(\xxxi) \widehat{\psi_k}(\xi_1)\sum_{k_2,\dots,k_m\leq k}\widehat{\psi_{k_2}}(\xi_2)\cdots\widehat{\psi_{k_m}}(\xi_m)}
\end{equation*}
and further decompose $\sigma^{(1)}$ as
\begin{equation*}
\sigma^{(1)}(\xxxi)=\sigma^{(1)}_{low}(\xxxi)+\sigma^{(1)}_{high}(\xxxi)
\end{equation*} where
\begin{equation*}
\sigma^{(1)}_{low}(\xxxi):=\sum_{k\in\mathbb{Z}}\sigma_k(\xxxi)\widehat{\psi_k}(\xi_1)\sum_{\substack{k_2,\dots,k_m\leq k\\ \max_{2\leq j\leq m}{(k_j)}\geq k-4-\lfloor \log_2{m}\rfloor}}{\widehat{\psi_{k_2}}(\xi_2)\cdots\widehat{\psi_{k_m}}(\xi_m)},
\end{equation*}
\begin{equation*}
\sigma^{(1)}_{high}(\xxxi):=\sum_{k\in\mathbb{Z}}\sigma_k(\xxxi)\widehat{\psi_k}(\xi_1)\sum_{k_2,\dots,k_m\leq k-5-\lfloor \log_2m\rfloor}{\widehat{\psi_{k_2}}(\xi_2)\cdots\widehat{\psi_{k_m}}(\xi_m)}.
\end{equation*} 
We refer to $T_{\sigma_{low}^{(1)}}$ as the low frequency part, and $T_{\sigma_{high}^{(1)}}$ as the high frequency part of $T_{\sigma^{(1)}}$ (due to the Fourier supports of summands in $T_{\sigma_{low}^{(1)}}\fff$ and $T_{\sigma_{high}^{(1)}}\fff$).

If $\sigma$  is compactly supported, like $\sigma_k$ in (\ref{sigmak}),  then   $\sigma^{\vee}$ exists, and $T_{\sigma}\big(f_1,\dots,f_m\big)(x)$ can be written as the convolution   $\sigma^{\vee}\ast \big( f_1\otimes\cdots\otimes f_m\big)(x,\dots,x)$. Then we may use the following lemma whose assertion is analogous to the key estimate in the proof of Theorem \ref{knownresult} in \cite{Gr_Sl}.
\begin{lemma}\label{keylemma1}
Let $\sigma$ be a bounded function on $\rr^{mn}$ such that $\sigma^{\vee}$ exists.
Suppose that $mn/2<s<mn$ and $q>mn/s$.
Then we have
\begin{equation}\label{keylemmaest}
\big| \sigma^{\vee}\ast \big(f_1\otimes\cdots \otimes f_m\big)(\xxx) \big|\lesssim \big\Vert \sigma(2^k\;\vec{\cdot}\;)\big\Vert_{L_s^{mn/s,1}(\rr^{mn})}\mathcal{M}_q^{(n)}f_1(x_1)\cdots \mathcal{M}_q^{(n)}f_m(x_m).
\end{equation}
\end{lemma}
\begin{proof}
Let $F(\xxx):=f_1(x_1)\cdots f_m(x_m)$.
Then the left-hand side of (\ref{keylemmaest}) is 
\begin{align*}
\big| \sigma^{\vee}\ast F(\xxx)\big|\leq \int_{\rr^{mn}}{\big(1+4\pi^2|\yyy|^2 \big)^{s/2}\big|\big(\sigma(2^k \;\vec{\cdot}\;) \big)^{\vee}(\yyy) \big| \frac{|F(\xxx-\yyy/2^k)|}{(1+4\pi^2|\yyy|^2)^{s/2}}}d\yyy
\end{align*} 
and this is bounded by
\begin{align*}
&\big\Vert \big( 1+4\pi^2|\;\vec{\cdot}\;|^2\big)^{s/2}  \big( \sigma(2^k\;\vec{\cdot}\;)\big)^{\vee}  \big\Vert_{L^{(mn/s)',1}(\rr^{mn})}\Big\Vert \frac{F(\xxx-\;\vec{\cdot}\;/2^k) }{(1+4\pi^2|\;\vec{\cdot}\;|^2)^{s/2}}\Big\Vert_{L^{mn/s,\infty}(\rr^{mn})}\\
&\lesssim \big\Vert \sigma(2^k\;\vec{\cdot}\;)\big\Vert_{L_s^{mn/s,1}(\rr^{mn})}\mathcal{M}_q^{(mn)}F(\xxx),
\end{align*} 
by applying Lemma \ref{hausyoung} with $mn/s>2$ and Lemma \ref{holder} with $s<mn$ and $q>mn/s$.
Note that every cube $Q$ in $\rr^{mn}$ containing $\xxx$ can be written as the product of $m$ cubes $Q_1,\dots,Q_m$ in $\rn$ such that $x_j\in Q_j$ for $1\leq j\leq m$, and $|Q|=|Q_1|\times\cdots\times |Q_m|$.
This implies that
\begin{equation*}
\mathcal{M}_q^{(mn)}F(\xxx)\leq \mathcal{M}_q^{(n)}f_1(x_1)\cdots \mathcal{M}_q^{(n)}f_m(x_m)
\end{equation*}
and therefore (\ref{keylemmaest}) follows.
\end{proof}

\subsubsection{Low frequency part}
To obtain the estimates for the operator $T_{\sigma^{(1)}_{low}}$,
we observe that
\begin{equation*}
T_{\sigma^{(1)}_{low}}\fff(x)=\sum_{k\in\mathbb{Z}}\sum_{\substack{k_2,\dots,k_m\leq k\\ \max_{2\leq j\leq m}{(k_j)}\geq k-4-\lfloor \log_2{m}\rfloor}} {T_{\sigma_k}\big( (f_1)_k,(f_2)_{k_2},\dots,(f_m)_{k_m}\big)(x)}
\end{equation*}
where $(g)_l:=\psi_l\ast g$ for $g\in \SS(\rn)$ and $l\in\mathbb{Z}$.
 It suffices to treat only the sum over $k-4-\lfloor \log_2{m}\rfloor \leq k_2\leq k$ and $k_3,\dots,k_m\leq k_2$, and we will actually prove that
\begin{align}\label{fgoal}
&\Big\Vert \sum_{k\in\mathbb{Z}}\sum_{\substack{k-4-\lfloor \log_2{m}\rfloor \leq k_2\leq k\\k_3,\dots,k_m\leq k_2}}T_{\sigma_k}\big((f_1)_k,(f_2)_{k_2},\dots,(f_m)_{k_m} \big)\Big\Vert_{L^p(\rn)}\nonumber\\
&\lesssim \sup_{k\in\zz}{\big\Vert \sigma (2^k\;\vec{\cdot}\;) \wh{\Theta^{(m)}}\big\Vert_{L_s^{mn/s,1}(\rr^{mn})}}\prod_{j=1}^{m}\Vert f_j\Vert_{L^{p_j}(\rn)}. 
\end{align}

Let $\phi$ be the Schwartz function on $\rn$ satisfying
\begin{equation*}
\wh{\phi}(\xi):=\begin{cases}
\sum_{j\leq 0}{\wh{\psi_j}}, & \xi\not= 0\\
1,& \xi=0
\end{cases}
\end{equation*}  and $\phi_j:=2^{jn}\phi(2^j\cdot)$.
 Then  for each $k \in\mathbb{Z}$  we   write
\begin{equation*}
\sum_{k_3,\dots,  k_m\leq k_2}T_{\sigma_k}\big((f_1)_k,(f_2)_{k_2},\dots,  (f_m)_{k_m} \big)(x)=T_{\sigma_k}\big((f_1)_k,(f_2)_{k_2},(f_3)^{k_2},\dots,   (f_m)^{k_2} \big)(x)
\end{equation*}
where  $(f_j)^{k_2}:=\phi_{k_2}\ast f_j$.  
Since the sum over $k_2$ in the left-hand side of (\ref{fgoal}) is a finite sum over $k_2$ near $k$, we may consider only the case $k_2=k$ and thus our claim is
\begin{align}\label{mainmainclaim}
&\Big\Vert \sum_{k\in\mathbb{Z}}{T_{\sigma_k}\big((f_1)_k,(f_2)_k,(f_3)^k,\dots,( f_m)^k \big)}\Big\Vert_{L^p(\rn)}\nonumber\\
&\lesssim \sup_{k\in\zz}{\big\Vert \sigma (2^k\;\vec{\cdot}\;) \wh{\Theta^{(m)}}\big\Vert_{L_s^{mn/s,1}(\rr^{mn})}}\prod_{j=1}^{m}{\Vert f_j\Vert_{L^{p_j}(\rn)}}.
\end{align}

To  prove the validity of (\ref{mainmainclaim}) we express
\begin{align*}
&T_{\sigma_k}\big((f_1)_k,(f_2)_k,(f_3)^k,\dots,(f_m)^k \big)(x)\\
&=(\sigma_k)^{\vee}\ast \big[(f_1)_k\otimes (f_2)_k\otimes (f_3)^k\otimes\cdots\otimes (f_m)^k \big](x,\dots,x),
\end{align*} 
and apply Lemma \ref{keylemma1} with $mn/s<q<p_1,\dots,p_m$. Then the preceding expression is dominated by a constant multiple of 
\begin{equation*}
\big\Vert \sigma_k(2^k\;\vec{\cdot}\;)\big\Vert_{L_s^{mn/s,1}(\rr^{mn})}\mathcal{M}_q^{(n)}(f_1)_k(x)\mathcal{M}_q^{(n)}(f_2)_k(x)\prod_{j=3}^{m}\mathcal{M}_q^{(n)}(f_j)^k(x),
\end{equation*}
and this yields that the left-hand side of (\ref{mainmainclaim}) is controlled by 
$$
\sup_{k\in\zz}\big\Vert \sigma(2^k\;\vec{\cdot}\;)\wh{\Theta^{(m)}}\big\Vert_{L_s^{mn/s,1}(\rr^{mn})} 
$$
 times
\begin{align*}
&\Big\Vert \sum_{k\in\zz}{\mathcal{M}_q^{(n)}(f_1)_k\mathcal{M}_q^{(n)}(f_2)_k\Big(\prod_{j=3}^{m}{\mathcal{M}_q^{(n)}(f_j)^k} \Big)}\Big\Vert_{L^p(\rn)}\\
&\leq \big\Vert \big\{ \mathcal{M}_q^{(n)}(f_1)_k\big\}_{k\in\zz}\big\Vert_{L^{p_1}(\ell^2)}\big\Vert \big\{ \mathcal{M}_q^{(n)}(f_2)_k\big\}_{k\in\zz}\big\Vert_{L^{p_2}(\ell^2)}\prod_{j=3}^{m}{\big\Vert \big\{ \mathcal{M}_q^{(n)}(f_j)^k\big\}_{k\in\zz}\big\Vert_{L^{p_j}(\ell^{\infty})}}\\
&\lesssim \big\Vert \big\{ (f_1)_k\big\}_{k\in\zz}\big\Vert_{L^{p_1}(\ell^2)}\big\Vert \big\{ (f_2)_k\big\}_{k\in\zz}\big\Vert_{L^{p_2}(\ell^2)}\prod_{j=3}^{m}{\big\Vert \big\{ (f_j)^k\big\}_{k\in\zz}\big\Vert_{L^{p_j}(\ell^{\infty})}}
\end{align*}
in view of H\"older's inequality and (\ref{fsmaximal}).
The well known equivalences  
\begin{equation*}
 \big\Vert \big\{ (f_1)_k\big\}_{k\in\zz}\big\Vert_{L^{p_1}(\ell^2)}\approx \Vert f_1\Vert_{L^{p_1}(\rn)},\quad  \big\Vert \big\{ (f_2)_k\big\}_{k\in\zz}\big\Vert_{L^{p_2}(\ell^2)}\approx \Vert f_2\Vert_{L^{p_2}(\rn)},
\end{equation*} and
\begin{equation*}
\big\Vert \big\{ (f_j)^k\big\}_{k\in\zz}\big\Vert_{L^{p_j}(\ell^{\infty})}\approx \Vert f_j\Vert_{H^{p_j}(\rn)}\approx \Vert f_j\Vert_{L^{p_j}(\rn)} 
\end{equation*}  conclude the proof of boundedness of  $T_{\sigma^{(1)}_{low}}$. 

\subsubsection{High frequency part}

The proof for the high frequency part relies on the fact that if $\widehat{g_k}$ is supported in $\{\xi \in\rn: C^{-1} 2^{k}\leq |\xi|\leq C2^{k}\}$ for $C>1$ then
\begin{equation}\label{marshall}
\Big\Vert \Big\{ \psi_k \ast \Big(\sum_{l=k-h}^{k+h}{g_l}\Big)\Big\}_{k\in\mathbb{Z}}\Big\Vert_{L^p(\ell^q)}\lesssim_{h,C} \big\Vert \big\{ g_k\big\}_{k\in\mathbb{Z}}\big\Vert_{L^p(\ell^q)}
\end{equation} 
for $h\in \mathbb{N}$. The proof of (\ref{marshall}) is     standard, and will not be pursued here. Just use the estimate $|\psi_k \ast g_l(x)|\lesssim \mathcal{M}^{(n)}g_l(x)$ for all $k-h\leq l\leq k+h$, and apply (\ref{fsmaximal}).

We note that
\begin{equation*}
T_{\sigma_{high}^{(1)}}\fff=\sum_{k\in\mathbb{Z}}{T_{\sigma_k}\big((f_1)_k,(f_2)^{k,m},\dots, (f_m)^{k,m} \big)},
\end{equation*}
where $\phi_j$ is defined as before and $(f_j)^{k,m}:=\phi_{k-5-\lfloor \log_2{m}\rfloor}\ast f_j$ for $2\leq j\leq m$.
Observe that the Fourier transform of $T_{\sigma_k}\big((f_1)_k,(f_2)^{k,m},\dots,(f_m)^{k,m} \big)$ is supported in $\big\{\xi\in\rn : 2^{k-2}\leq |\xi|\leq 2^{k+2} \big\}$ 
and thus (\ref{marshall}) yields that
\begin{equation*}
\big\Vert T_{\sigma_{high}^{(1)}}\fff \big\Vert_{L^p(\rn)}\lesssim \big\Vert \big\{  T_{\sigma_k}\big((f_1)_k,(f_2)^{k,m},\dots,(f_m)^{k,m} \big)\big\}_{k\in\mathbb{Z}}\big\Vert_{L^p(\ell^{2})}.
\end{equation*}

Now we write, as before,
\begin{align*}
&T_{\sigma_k}\big((f_1)_k,(f_2)^{k,m},\dots,(f_m)^{k,m} \big)(x)=(\sigma_k)^{\vee}\ast \big[(f_1)_k\otimes (f_2)^{k,m}\otimes\cdots\otimes (f_m)^{k,m} \big](x,\dots,x)
\end{align*} 
and use Lemma \ref{keylemma1} with $mn/s<q<p_1,\dots,p_m$ to obtain
\begin{align*}
&\big|T_{\sigma_k}\big((f_1)_k,(f_2)^{k,m},\dots,(f_m)^{k,m} \big)(x) \big|\\
&\lesssim \big\Vert \sigma_k(2^k\;\vec{\cdot}\;)\big\Vert_{L_s^{mn/s,1}(\rr^{mn})} \mathcal{M}_q^{(n)}(f_1)_k(x)\prod_{j=2}^{m}{\mathcal{M}_q^{(n)}(f_j)^{k,m}(x)}.
\end{align*}
Therefore, $\big\Vert T_{\sigma_{high}^{(1)}}\fff \big\Vert_{L^p(\rn)}$ is controlled by a constant times
\begin{align*}
  \sup_{k\in\zz}\big\Vert \sigma(2^k\;\vec{\cdot}\;)\wh{\Theta^{(m)}}\big\Vert_{L_s^{mn/s,1}(\rr^{mn})}\Big\Vert \Big\{ \mathcal{M}_q^{(n)}(f_1)_k(x)\prod_{j=2}^{m}{\mathcal{M}_q^{(n)}(f_j)^{k,m}(x)} \Big\}_{k\in\zz}\Big\Vert_{L^p(\ell^2)}.
\end{align*}
We now apply H\"older's inequality and  (\ref{fsmaximal}) to show that
\begin{align*}
&\Big\Vert \Big\{ \mathcal{M}_q^{(n)}(f_1)_k \prod_{j=2}^{m}{\mathcal{M}_q^{(n)}(f_j)^{k,m}} \Big\}_{k\in\zz}\Big\Vert_{L^p(\ell^2)}\\
&\lesssim \big\Vert \big\{  \mathcal{M}_q^{(n)}(f_1)_k \big\}_{k\in\zz} \big\Vert_{L^{p_1}(\ell^2)}\prod_{j=2}^{m}{ \big\Vert \big\{  \mathcal{M}_q^{(n)}(f_j)^{k,m} \big\}_{k\in\zz} \big\Vert_{L^{p_j}(\ell^{\infty})}}\\
&\lesssim \big\Vert \big\{  (f_1)_k \big\}_{k\in\zz} \big\Vert_{L^{p_1}(\ell^2)}\prod_{j=2}^{m}{ \big\Vert \big\{ (f_j)^{k,m} \big\}_{k\in\zz} \big\Vert_{L^{p_j}(\ell^{\infty})}}\\
&\approx \Vert f_1\Vert_{L^{p_1}(\rn)}\prod_{j=2}^{m}{\Vert f_j\Vert_{H^{p_j}(\rn)}}\approx \prod_{j=1}^{m}{\Vert f_j\Vert_{L^{p_j}(\rn)}}.
\end{align*}
This completes the proof of the case $mn/s< p_1,\dots,p_m<\infty$.

\subsection{The case $1<p,p_1,\dots,p_m<\infty$}

Let $T_{\sigma}^{*j}$ be the $j$th transpose of $T_{\sigma}$, defined as the unique operator satisfying
\begin{equation*}
\big\langle T_{\sigma}^{*j}(f_1,\dots,f_m),h\big\rangle:=\big\langle T_{\sigma}(f_1,\dots,f_{j-1},h,f_{j+1},\dots,f_m),f_j\big\rangle
\end{equation*}
for $f_1,\dots,f_m,h\in \SS(\rn)$.
Observe that $T_{\sigma}^{*j}=T_{\sigma^{*j}}$ where 
\begin{equation*}
\sigma^{*j}(\xi_1,\dots,\xi_m)=\sigma\big(\xi_1,\dots,\xi_{j-1},-(\xi_1+\dots+\xi_m),\xi_{j+1},\dots,\xi_m\big)
\end{equation*}
and we claim that for any $1\leq j\leq m$
\begin{equation}\label{multiestimate}
\sup_{k\in\zz}\big\Vert \sigma^{*j}(2^k\;\vec{\cdot}\; )\wh{\Psi^{(m)}}\big\Vert_{L_{s}^{mn/s,1}(\rr^{mn})}\lesssim \sup_{k\in\zz}\big\Vert \sigma(2^k\;\vec{\cdot}\; )\wh{\Psi^{(m)}}\big\Vert_{L_{s}^{mn/s,1}(\rr^{mn})}.
\end{equation}
To see this, we need the following lemma:
\begin{lemma}\label{lorentzad}
Let $1<p<\infty$, $0<q\leq \infty$, and $s\geq 0$.
Let $f\in \SS(\rr^{mn})$ and for each $1\leq j\leq m$ let
\begin{equation*}
T^jf(x_1,\dots,x_m):=f(x_1,\dots, x_{j-1},-(x_1+\dots +x_m),x_{j+1},\dots ,x_m)
\end{equation*} for $x_1,\dots,x_m\in\rn$.
Then every $T^j$ satisfies the estimate
\begin{equation}\label{lorentzadjoint}
\big\Vert T^jf\big\Vert_{L_s^{p,q}(\rr^{mn})}\lesssim \big\Vert f\big\Vert_{L_s^{p,q}(\rr^{mn})}.
\end{equation}
\end{lemma}
\begin{proof}
It is enough to deal only with the case $j=1$ because the other cases will follow from a symmetric argument.

{\bf Step 1.} 
We claim that for $k\in\{0,1,2,\dots\}$
\begin{equation}\label{lpadjoint}
\big\Vert \big(\vec{I}-\vec{\Delta}\big)^{k}T^1f \big\Vert_{L^p(\rr^{mn})}\lesssim \big\Vert \big(\vec{I}-\vec{\Delta}\big)^kf\big\Vert_{L^p(\rr^{mn})}.
\end{equation}
Using   Leibniz's rule we write 
\begin{align*}
\big|\big(\vec{I}-\vec{\Delta}\big)^kT^1f(x_1,\dots,x_m) \big|&\approx \Big| \sum_{l=0}^{k}{c_l (-\vec{\Delta})^lT^1f(x_1,\dots,x_m)}\Big|\\
&\approx \Big| \sum_{l=0}^{k}{d_l \big[(-\vec{\Delta})^lf\big]\big(-(x_1+ \cdots  +x_m ), x_2,\dots , x_m\big)}\Big|\\
&\approx \big| \big[ \big(\vec{I}-\vec{\Delta}\big)^kf\big]\big(-(x_1+  \cdots  + x_m ),x_2,\dots,x_m\big)\big|,
\end{align*} 
for some     constants $c_l,d_l$.
Then (\ref{lpadjoint}) can be achieved through a change of variables in $L^p$.

{\bf Step 2.} 
From Step 1, $T^1$ is a linear operator 
\begin{equation*}
T^1:L_{2k}^{p}(\rr^{mn})\to L_{2k}^{p}(\rr^{mn})
\end{equation*} for all $1<p<\infty$ and $k\in \{0,1,2,\dots\}$.
We perform a complex interpolation method with the fact that $\big(L_{s_0}^p(\rr^{mn}),L_{s_1}^p(\rr^{mn})\big)_{\theta}=L_s^p(\rr^{mn})$ for $s=(1-\theta)s_0+\theta s_1$, and then obtain that
\begin{equation}\label{sobolevadjoint}
\big\Vert T^1f\Vert_{L_s^p(\rr^{mn})}\lesssim \Vert f\Vert_{L_s^p(\rr^{mn})}
\end{equation} for all $s\geq 0$ and $1<p<\infty$.

{\bf Step 3.} 
Let $0<q\leq \infty$ and $s\geq 0$. We define a linear operator $T^{1,s}$ by
\begin{equation*}
T^{1,s}f(x_1,\dots,x_m):=\big(\vec{I}-\vec{\Delta}\big)^{s/2}\big[T^1\big(\vec{I}-\vec{\Delta}\big)^{-s/2}f \big](x_1,\dots,x_m).
\end{equation*}
Then (\ref{sobolevadjoint}) implies that for all $1<p<\infty$
\begin{equation*}
\big\Vert T^{1,s}f\big\Vert_{L^p(\rr^{mn})}\lesssim \Vert f\Vert_{L^p(\rr^{mn})}.
\end{equation*}
Using   real interpolation with $\big(L^{p_0}(\rr^{mn}),L^{p_1}(\rr^{mn})\big)_{\theta,q}=L^{p,q}(\rr^{mn})$ for $1/p=(1-\theta)/p_0+\theta/p_1$, we obtain  
\begin{equation*}
\big\Vert T^{1,s}f \big\Vert_{L^{p,q}(\rr^{mn})}\lesssim \Vert f\Vert_{L^{p,q}(\rr^{mn})},
\end{equation*}
which is equivalent to (\ref{lorentzadjoint}).
\end{proof}

Now let us prove (\ref{multiestimate}). Lemma \ref{lorentzad} yields that 
\begin{align*}
&\sup_{k\in\zz}\big\Vert \sigma^{*j}(2^k\;\vec{\cdot}\; )\wh{\Psi^{(m)}}\big\Vert_{L_{s}^{mn/s,1}(\rr^{mn})}\\
&~~\lesssim \sup_{k\in\zz}\big\Vert \sigma(2^k\xxxi )\wh{\Psi^{(m)}}\big(\xi_1,\dots,\xi_{j-1},-(\xi_1+\dots+\xi_m),\xi_{j+1},\dots,\xi_m\big)\big\Vert_{L_{s}^{mn/s,1}(\xxxi)}.
\end{align*}
Since $$\frac{1}{\sqrt{3}}|\xxxi|\leq  \big(|\xi_1|^2+\dots+|\xi_{j-1}|^2+|\xi_1+\dots+\xi_m|^2+|\xi_{j+1}|^2+\dots+|\xi_m|^2 \big)^{1/2}\leq \sqrt{3}|\xxxi|,$$
the preceding expression can be written as
\begin{equation*}
\sup_{k\in\zz}\Big\Vert \sigma(2^k\xxxi )\wh{\Lambda^{(m)}}(\xxxi)\wh{\Psi^{(m)}}\big(\xi_1,\dots,
 \xi_{j-1},-(\xi_1+\dots+\xi_m),\xi_{j+1},\dots,\xi_m\big)\Big\Vert_{L_{s}^{mn/s,1}(\xxxi)}
\end{equation*}
where $\Lambda^{(m)}$ is a Schwartz function on $\rr^{mn}$ having the properties that
$\wh{\Lambda^{(m)}}$ is supported in the annulus $2^{-2}\leq |\xxxi|\leq 2^2$ and $\wh{\Lambda^{(m)}}(\xxxi)=1$ for $\frac{1}{2\sqrt{3}}\leq|\xxxi|\leq 2\sqrt{3}$. 
Using Lemma \ref{katoponce}, the supremum is controlled by a constant multiple of
 \begin{equation*}
\sup_{k\in\zz}\big\Vert \sigma(2^k\;\vec{\cdot}\; )\wh{\Lambda^{(m)}}\big\Vert_{L_{s}^{mn/s,1}(\rr^{mn})}
\end{equation*} 
and we obtain (\ref{multiestimate}) in the same way as   (\ref{reproducing}).

Now we complete the proof. Assume $1<p\leq \min{(p_{1},\dots, p_m)}<2$ (otherwise, we are done from Section \ref{firstcase}). Observe that only one of $p_j$ could be less than $2$ because $1/p=1/p_1+\dots+1/ p_m<1$, and we will actually look at the case $1<p_1<2\leq p_2,\dots, p_m$.
Let $2<p',p_1'<\infty$ be the H\"older conjugates of $p, p_1$, respectively. That is, $1/p+1/p'=1/p_1+1/p_1'=1$ and accordingly, $1/p_1'=1/p'+1/p_2+\dots+1/p_m$ and $2\leq p',p_2,\dots,p_m<\infty$.
Finally, we have
\begin{align*}
\big\Vert T_{\sigma }(\fff)\big\Vert_{L^p(\rn)} &= \sup_{\Vert h\Vert_{L^{p'}(\rn)}=1}{\big|\big\langle T_{\sigma^{*1}}(h,f_2,\dots,f_m),f_1\big\rangle \big|}\\
&\leq \Vert f_1\Vert_{L^{p_1}(\rn)} \sup_{\Vert h\Vert_{L^{p'}(\rn)}=1}{\big\Vert T_{\sigma^{*1}}(h,f_2,\dots,f_m)\big\Vert_{L^{p_1'}(\rn)} }\\
&\lesssim \sup_{k\in\zz}\big\Vert \sigma(2^k\;\vec{\cdot}\; )\Psi^{(m)}\big\Vert_{L_{s}^{mn/s,1}(\rr^{mn})}\prod_{j=1}^{m}\Vert f_j\Vert_{L^{p_{j}}(\rn)}
\end{align*}
where (\ref{multiestimate}) is applied.

\section{Proof of Theorem \ref{main2}}\label{proofmain2}

For any $0<t,\gamma<\infty$ we define 
\begin{equation*}
\mathcal{H}_{(t,\gamma)}(\xxx):=\frac{1}{(1+4\pi^2|\xxx|^2)^{t/2}}\frac{1}{(1+\ln{(1+4\pi^2|\xxx|^2)})^{\gamma/2}}.
\end{equation*}
We first see that 
\begin{equation}\label{hproperty2}
\Vert \mathcal{H}_{(t,\gamma)}\Vert_{L^r(\rr^{mn})}<\infty \quad \text{if and only if}\quad  t>mn/r \quad \text{or} \quad t=mn/r, \gamma>2/r.
\end{equation}
Moreover, it was shown in \cite{Gr_Park}  that
\begin{equation*}
\big| \wh{\mathcal{H}_{(t,\gamma)}}(\xxxi)\big|\lesssim_{t,\gamma,n,m} e^{-|\xxxi|/2}  \quad \text{for }~ |\xxxi|>1
\end{equation*}
and when $0<t<mn$,
\begin{equation*}
\big| \wh{\mathcal{H}_{(t,\gamma)}}(\xxxi)\big|\approx_{t,\gamma,n,m}|\xxxi|^{-(mn-t)}(1+2\ln|\xxxi|^{-1})^{-\gamma/2} \quad \text{for }~ |\xxxi|\leq 1.
\end{equation*}
The estimates imply that
\begin{equation}\label{hproperty4}
\big\Vert \wh{\mathcal{H}_{(t,\gamma)}}\big\Vert_{L^{r,q}(\rr^{mn})}<\infty ~ \text{ if and only if } ~ t>mn-mn/r \quad \text{or}\quad t=mn-mn/r, \gamma>2/q.
\end{equation} 

Based on the properties of $\mathcal{H}_{t,\gamma}$, let us construct counter examples to prove Theorem \ref{main2}.
Let $\Gamma$ denote a Schwartz function on $\rr^{mn}$ such that $\textup{Supp}(\wh{\Gamma})\subset \{\xxxi\in \rr^{mn}:\frac{99}{100}\leq |\xxxi|\leq \frac{101}{100}\}$ and $\wh{\Gamma}(\xxxi)=1$ for $\frac{999}{1000}\leq |\xxxi|\leq \frac{1001}{1000}$.
Let $N$ be a sufficiently large positive interger and we define
\begin{equation*}
\mathcal{H}_{(t,\gamma)}^{(N)}(\xxx):=\mathcal{H}_{(t,\gamma)}(\xxx)\wh{\Phi_N}(\xxx), \qquad \xxx\in \rr^{mn}
\end{equation*}
and 
\begin{equation*}
\sigma^{(N)}(\xxxi):=\wh{\mathcal{H}_{(t,\gamma)}^{(N)}}(\xxxi)\wh{\Gamma}(\xxxi), \qquad \xxxi\in\rr^{mn}.
\end{equation*}
Then $\sigma^{(N)}$ is supported in $\{\xxxi\in\rr^{mn}:\frac{99}{100}\leq |\xxxi|\leq \frac{101}{100}\}$ in view of the support of $\wh{\Gamma}$, and this implies that
$\sigma^{(N)}(2^k\xxxi)\wh{\Psi^{(m)}}(\xxxi)$ vanishes unless $-1\leq k\leq 1$. 
Therefore, using Lemma \ref{katoponce} and a scaling argument, we have
\begin{align*}
\sup_{k\in\zz}\big\Vert \sigma^{(N)}(2^k\;\vec{\cdot}\;)\wh{\Psi^{(m)}}\big\Vert_{L_s^{r,q}(\rr^{mn})}&=\max_{-1\leq k\leq 1}{\big\Vert \sigma^{(N)}(2^k\;\vec{\cdot}\;)\wh{\Psi^{(m)}}\big\Vert_{L_s^{r,q}(\rr^{mn})}}\\
&\lesssim \max_{-1\leq k\leq 1}\big\Vert \sigma^{(N)}(2^k\;\vec{\cdot}\;)\big\Vert_{L_s^{r,q}(\rr^{mn})}\lesssim \Vert \sigma^{(N)}\Vert_{L_s^{r,q}(\rr^{mn})}.
\end{align*}
This can be further estimated, using Lemma \ref{katoponce}, by a constant times
\begin{equation*}
\big\Vert \wh{\mathcal{H}_{(t,\gamma)}^{(N)}}\big\Vert_{L_s^{r,q}(\rr^{mn})}=\big\Vert \Phi_N\ast \wh{\mathcal{H}_{(t-s,\gamma)}}\big\Vert_{L^{r,q}(\rr^{mn})}
\end{equation*}
where the equality follows from fact that
\begin{equation*}
(\vec{I}-\vec{\Delta})^{s/2}\wh{\mathcal{H}^{(N)}_{(t,\gamma)}}(\xxxi)=\wh{\mathcal{H}^{(N)}_{(t-s,\gamma)}}(\xxxi)=\Phi_N\ast \wh{\mathcal{H}_{(t-s,\gamma)}}(\xxxi).
\end{equation*} 
Finally, Lemma \ref{young} yields that
\begin{equation}\label{upperbound}
\sup_{k\in\zz}\big\Vert \sigma^{(N)}(2^k\;\vec{\cdot}\;)\wh{\Psi^{(m)}}\big\Vert_{L_s^{r,q}(\rr^{mn})}\lesssim \big\Vert \wh{\mathcal{H}_{(t-s,\gamma)}}\big\Vert_{L^{r,q}(\rr^{mn})},\quad \text{ uniformly in }~ N.
\end{equation}

On the other hand, for $0<\epsilon<1/100$ and for each $0<p_j\leq \infty$, let 
\begin{equation*}
f_j^{(\epsilon)}(x):=\epsilon^{n/p_j}\theta(\epsilon x),\qquad 1\leq j\leq m
\end{equation*} where $\theta$ is a Schwartz function on $\rn$ with $\textup{Supp}(\wh{\theta})\subset \{\xi\in\rn:\frac{999}{1000\sqrt{m}}\leq |\xi|\leq \frac{1001}{1000\sqrt{m}}\}$. 
Clearly, we have
\begin{equation}\label{fjest}
\Vert f_j^{(\epsilon)}\Vert_{L^{p_j}(\rn)}=\Vert \theta\Vert_{L^{p_j}(\rn)}\lesssim_{p_j,n} 1 \qquad \text{uniformly in }~ \epsilon.
\end{equation}
In addition,
\begin{equation*}
T_{\sigma^{(N)}}\big( f_1^{(\epsilon)},\dots,f_m^{(\epsilon)}\big)(x)=\epsilon^{n/p}\mathcal{H}_{(t,\gamma)}^{(N)}\ast \big(\theta(\epsilon\cdot)\otimes\cdots\otimes\theta(\epsilon\cdot) \big)(x,\dots,x)
\end{equation*} since $\wh{\Gamma}=1$ on the support of the Fourier transform of $\theta(\epsilon\cdot)\otimes\cdots\otimes\theta(\epsilon\cdot) $.
By taking the $L^p$ norm and using a scaling, we see that
\begin{align*}
&\big\Vert  T_{\sigma^{(N)}}\big( f_1^{(\epsilon)},\dots,f_m^{(\epsilon)}\big)\big\Vert_{L^p(\rn)}\\
&=\Big(\int_{\rn}{\Big|\int_{\rr^{mn}}{\mathcal{H}_{(t,\gamma)}^{(N)}(\yyy)\theta(x-\epsilon y_1)\cdots \theta(x-\epsilon y_m)}d\yyy \Big|^p}dx \Big)^{1/p}.
\end{align*}
We now apply (\ref{fjest}) and Fatou's lemma to obtain
\begin{align}\label{lowerbound}
&\Vert T_{\sigma^{(N)}}\Vert_{L^{p_1}\times\cdots\times L^{p_m}\to L^p}\nonumber\\
&\gtrsim \liminf_{\epsilon\to 0}\big\Vert  T_{\sigma^{(N)}}\big( f_1^{(\epsilon)},\dots,f_m^{(\epsilon)}\big)\big\Vert_{L^p(\rn)}\nonumber\\
&\geq  \Big(\int_{\rn}{\Big|\liminf_{\epsilon\to 0}\int_{\rr^{mn}}{\mathcal{H}_{(t,\gamma)}^{(N)}(\yyy)\theta(x-\epsilon y_1)\cdots \theta(x-\epsilon y_m)}d\yyy \Big|^p}dx \Big)^{1/p}.
\end{align}
Since 
\begin{equation*}
\big| \mathcal{H}_{(t,\gamma)}^{(N)}(\yyy)\theta(x-\epsilon y_1)\cdots \theta(x-\epsilon y_m)\big|\lesssim \big| \mathcal{H}_{(t,\gamma)}^{(N)}(\yyy)\big| \quad \text{ uniformly in }~\epsilon>0, ~ x\in\rn
\end{equation*}
and
\begin{equation*}
\big\Vert \mathcal{H}_{(t,\gamma)}^{(N)}\big\Vert_{L^1(\rr^{mn})}\leq \Vert \wh{\Phi_N}\Vert_{L^1(\rr^{mn})}\lesssim N^{mn}<\infty,
\end{equation*}
we can utilize the Lebesgue dominated convergence theorem and then
\begin{equation*}
(\ref{lowerbound})=\big\Vert |\theta|^m\big\Vert_{L^p(\rn)}\big\Vert \mathcal{H}_{(t,\gamma)}^{(N)}\big\Vert_{L^1(\rr^{mn})}\approx \big\Vert \mathcal{H}_{(t,\gamma)}\wh{\Phi_N}\big\Vert_{L^1(\rr^{mn})}.
\end{equation*}
By taking $\liminf_{N\to\infty}$, we finally obtain that
\begin{equation}\label{finallowerbound}
\liminf_{N\to\infty}\Vert T_{\sigma^{(N)}}\Vert_{L^{p_1}\times\cdots\times L^{p_m}\to L^p}\gtrsim\Vert \mathcal{H}_{(t,\gamma)}\Vert_{L^1(\rr^{mn})}
\end{equation} where the monotone convergence theorem is applied.

If $0<r<mn/s$ and $0<q\leq \infty$, then choose $t$ satisfying 
$$mn-(mn/r-s)<t<mn,$$
which is equivalent to $mn-mn/r<t-s<mn$ and $t<mn$. It follows from (\ref{upperbound}) and (\ref{hproperty4}) that
\begin{equation*}
\limsup_{N\to\infty}\sup_{k\in\zz}{\big\Vert \sigma^{(N)}(2^k\cdot)\wh{\Psi^{(m)}}\big\Vert_{L_s^{r,q}(\rr^{mn})}}<\infty,
\end{equation*} and (\ref{finallowerbound}) and (\ref{hproperty2}) yield that
\begin{equation*}
\liminf_{N\to\infty}\Vert T_{\sigma^{(N)}}\Vert_{L^{p_1}\times\cdots\times L^{p_m}\to L^p}=\infty,
\end{equation*} which proves the first assertion of Theorem \ref{main2}

Now assume $r=mn/s$ and $q>1$ and let $\gamma$ be a positive number with $2/q<\gamma\leq 2$. 
Then thanks to (\ref{hproperty2}), (\ref{hproperty4}), (\ref{upperbound}), and (\ref{finallowerbound}), we have 
\begin{equation*}
\limsup_{N\to\infty}\sup_{k\in\zz}{\big\Vert \sigma^{(N)}(2^k\cdot)\wh{\Psi^{(m)}}\big\Vert_{L_s^{mn/s,q}(\rr^{mn})}}<\infty
\end{equation*} and
\begin{equation*}
\liminf_{N\to\infty}\Vert T_{\sigma^{(N)}}\Vert_{L^{p_1}\times\cdots\times L^{p_m}\to L^p}=\infty,
\end{equation*}
which completes the proof of Theorem \ref{main2}.

\end{document}